\documentclass[10pt,a4paper]{amsart}

\usepackage[utf8]{inputenc}
\usepackage[english]{babel}

\usepackage{latexsym}
\usepackage{amsfonts}
\usepackage{amsmath}
\usepackage{amssymb}
\usepackage{graphicx}
\usepackage{color}
\usepackage{enumerate}

\usepackage{amsthm}

\newcommand{\HS}{\mathbb{R}^3_{\scalebox{1.0}[1.0]{\,-}}}
\newcommand{\CHS}{\overline{\mathbb{R}^3_{\scalebox{0.5}[1.0]{\( - \)}}}}

\newcommand{\Bm}{B_{\scalebox{0.5}[1.0]{\( - \)}}}

\newcommand{\R}{\mathbb{R}}

\newcommand{\C}{\mathbb{C}}

\newcommand{\nuoli}{\rightarrow}
\newcommand{\Lc}{L_{comp}}
\newcommand{\Hc}{H_{comp}}
\newcommand{\Wc}{W_{comp}}

\newcommand{\Curl}{\nabla \times}

\newcommand{\ov}[1]{\overline{#1}}

\renewcommand\Re{\operatorname{Re}}
\renewcommand\Im{\operatorname{Im}}
\newcommand\supp{\operatorname{supp}}

\newcommand\p{\operatorname{\partial}} 

\newcommand{\Ld}{L_{A,q}}

\newtheorem{thm}{Theorem}[section]
\newtheorem{cor}[thm]{Corollary}
\newtheorem{lem}[thm]{Lemma}
\newtheorem{prop}[thm]{Proposition}

\newtheorem{rem}[thm]{Remark}

\theoremstyle{definition}

\numberwithin{equation}{section}

\begin{document}

\title[]{An Inverse problem for the Magnetic Schr\"odinger  Operator on a Half Space with partial data}
\author[]{Valter Pohjola}
\date{November , 2012}
\address{Department of Mathematics and Statistics, Helsingin yliopisto / Helsingfors universitet / University of Helsinki, Finland}
\email{valter.pohjola@helsinki.fi}

\begin{abstract} In this paper we prove uniqueness for an inverse boundary value problem 
    for the magnetic Schr\"odinger equation in a half space, with partial data. 
We prove that the curl of the magnetic potential $A$, when $A\in \Wc^{1,\infty}(\CHS,\R^3)$, and 
the electric pontetial $q \in \Lc^{\infty}(\CHS,\C)$
are uniquely determined by the knowledge
of the Dirichlet-to-Neumann map on parts of the boundary of the half space.
\end{abstract}

\maketitle

\section{Introduction}

We consider a magnetic Schrödinger operator $\Ld$, defined by
\begin{eqnarray}
\Ld := \sum_{j=1}^3 \big(-i\partial_j + A_j(x)\big)^2 + q(x), \label{eq:Ldexp}
\end{eqnarray}
on the half space $\HS:=\{ x \in \R^3 \;:\; x_3 < 0 \}$. We assume that 
\begin{align}
A \in \Wc^{1,\infty}(\CHS,\R^3),   \text{ and }
q \in \Lc^{\infty}(\CHS,\C),   \quad \Im q\le 0.
\end{align}
Here 
\[
\Wc^{1,\infty}(\CHS,\R^3):=\{A|_{\CHS}  \; :\, A\in W^{1,\infty}(\R^3, \R^3), \supp(A)\subset 
\R^3 \text{ compact}\}
\]
and similarly, we define
\[
\Lc^{\infty}(\CHS,\C):=\{q\in L^{\infty}(\CHS,\C)\;:\,\supp(q)\subset 
\CHS \text{ compact}\}.
\]

Consider the Dirichlet problem
\begin{equation}
 \left.\begin{aligned}
        (\Ld - k^2) u & = 0 \quad \text{ in $\HS$,}\\
        u |_{\partial \HS} &= f,
       \end{aligned}
 \right.
 \label{eq:BVP}
\end{equation}
where $k>0$ is fixed and  $f \in \Hc^{3/2}(\p\HS)$. 
Here we also require that the solution $u$ should satisfy a 
boundary
condition at infinity, which will be the
\textit{Sommerfeld radiation condition}
\begin{equation}
\lim_{|x| \nuoli \infty} |x| \bigg( \frac{\partial u(x)}{ \partial |x|}
-iku(x)\bigg) = 0.
\label{eq:SRC} 
\end{equation}
Solutions satisfying this condition are called \textit{outgoing}
or \textit{radiating} solutions. We will also occasionally use the term
\textit{incoming} solution. This refers to a solution of \eqref{eq:BVP} that 
satisfies 
\eqref{eq:SRC}, when the factor $-ik$ is replaced by $ik$.
The existence and uniqueness of a solution $u \in H^2_{loc}(\overline{\HS})$ to 
the problem \eqref{eq:BVP} and \eqref{eq:SRC} is proven in \cite{Poh1}.
This allows us to define the so called \textit{Dirichlet 
to Neumann map} $\Lambda_{A,q}$,
(DN-map for short),
$\Lambda_{A,q}:\Hc^{3/2}(\p \HS) \nuoli H^{1/2}_{loc}(\p \HS)$ as
\begin{align*}
f\mapsto(\partial_n + i A\cdot n) u |_{\p \HS},
\end{align*}
where $u$ is the solution of the Dirichlet problem \eqref{eq:BVP}, 
\eqref{eq:SRC} and $f$ is the value of $u$. 
Here $n=(0,0,1)$ is the unit outer normal to the boundary $\p \HS$.

The inverse problem is to investigate if the DN-map uniquely determines the
potentials $A$ and $q$ in $\HS$. It turns out that the DN-map does not
in general uniquely determine $A$. This is due to the  gauge invariance of
the DN-map, which was first noticed by \cite{S1}.
\begin{lem} \label{gauge_inv} 
Let $A \in W^{1,\infty}(\CHS,\R^3)$ and $q \in L^{\infty}(\CHS)$. Then
\begin{enumerate}[(i)]
\item For all $\psi \in C^{1,1}(\CHS,\R)$ we  have 
\[
e^{-i \psi}\Lambda_{A,q} e^{i\psi} =\Lambda_{A+\nabla\psi,q}.
\]
\item  There exists $\psi \in C^{1,1}(\CHS,\R)$ with $\psi|_{\{x_3=0\}} = 0$,
for which
\[
\Lambda_{A,q} = \Lambda_{A+\nabla\psi,q}
\]
and $(A+\nabla\psi)|_{\{x_3=0\}} = (A_1,A_2,0)$.
\end{enumerate}
\end{lem}
\begin{proof}  See \cite{Poh1}.
\end{proof}
Part \textit{(ii)} of this Lemma shows that $\Lambda_{A,q}$ cannot uniquely
determine $A$, since we can change a potential by a gauge transformation without
changing the DN-map.
The DN-map does however carry enough information to determine
$\nabla\times A$, which is the magnetic field in the context of electrodynamics. 

When considering a pair of magnetic potentials $A_j$, $j=1,2$,
we use the notation  $A_j=(A_{j,1},A_{j,2},A_{j,3})$ for the component
functions.
Furthermore we let $N \subset \CHS$ be a relatively open and bounded 
set, for which
\begin{align*}
    \bigcup_{j=1,2} \supp(A_j) \cup \supp(q_j) \subset N,
\end{align*}
and for which $\p N$ piecewise $C^2$ and $\HS\setminus \ov{N}$ is connected.

We now state the main result of this paper, which generalizes the 
corresponding results of \cite{LCU1}, obtained in the case of the  
Schr\"odinger operator without a magnetic potential. 
\begin{thm}
\label{thm_2_inverse}
Let $A_j\in \Wc^{1,\infty}(\CHS,\R^3)$ and $q_j \in \Lc^{\infty}(\CHS,\C)$ be 
such that  $\Im q_j\le 0$, $j=1,2$. 
Let $\Gamma_1,\Gamma_2\subset\p \R^3_-$ be open sets such that
\begin{align}
(\p \HS \setminus \ov{N}) \cap \Gamma_j \neq \emptyset, \quad j=1,2.
\label{eq_gamma}
\end{align}
Then if 
\begin{equation}
\label{eq_data_inv}
\Lambda_{A_1,q_1}(f)|_{\Gamma_1}=\Lambda_{A_2,q_2}(f)|_{\Gamma_1},
\end{equation}
for all $f\in \Hc^{3/2} (\p \HS)$, $\supp(f)\subset 
\Gamma_2$, then 
\[
\nabla\times A_1=\nabla \times A_2\quad\text{and}\quad q_1=q_2\quad 
\text{in}\quad \HS.
\]
\end{thm}

We would like to emphasize that in Theorem \ref{thm_2_inverse}, the set 
$\Gamma_1$, where measurements are performed, and the set $\Gamma_2$, where the 
data is supported,  can be taken arbitrarily small, provided that 
\eqref{eq_gamma} holds.  The result of Theorem \ref{thm_2_inverse}  pertains 
therefore to the inverse problems with partial data. Such problems are important 
from the point of view of applications, since in practice, performing 
measurements on the entire boundary could be impossible, due to limitations in 
resources or obstructions from obstacles.  

The first uniqueness result, 
in the context of inverse boundary value problems for the magnetic 
Schr\"o\-dinger operator on a bounded domain, was 
obtained by Sun in \cite{S1}, under a smallness condition on $A$. 
Nakamura, Sun and Uhlmann proved the uniqueness without any smallness condition 
in \cite{NSU1}, assuming that $A \in C^2$. Tolmasky extended this result to
$C^1$ magnetic potentials in \cite{Tol1}, and Panchenko to some less regular but small
magnetic potentials in \cite{Pan1}.
Salo proved uniqueness for Dini continuous magnetic potentials in \cite{MS1}. 
The most recent result is given by Krupchyk and Uhlmann in \cite{Kru1}, where
uniqueness is proved for $L^\infty$ magnetic potentials.
In all of these works, the inverse boundary value 
problem with full data was considered. 

In \cite{ER1}, Eskin and Ralston obtained a uniqueness result for the closely 
related inverse scattering problem, assuming the exponential decay of the 
potentials.
The partial data problem in the magnetic case was considered by Dos Santos 
Ferreira, Kenig, Sjöstrand and Uhlmann in \cite{DosSantos1} and by Chung 
in \cite{Chung1}.

The inverse problem for the half space geometry, without a magnetic potential 
was examined by Cheney and Isaacson in \cite{CI1}. The uniqueness for this 
problem in the case of compactly supported electric potentials was proved
by Lassas, Cheney and Uhlmann in \cite{LCU1}, assuming that the supports do not 
come close to the boundary of the half space. The result of  Theorem 
\ref{thm_2_inverse} is therefore already a generalization of the work  
\cite{LCU1}, even in the absence of magnetic potentials. Li and Uhlmann  proved 
uniqueness for the closely related infinite slab geometry with $A=0$, in 
\cite{LU1}. Krupchyk, Lassas and Uhlmann did this for the magnetic case in
\cite{KLU1}. In both of these works, the reflection argument of Isakov 
\cite{I1} played an important role. The uniqueness problem for the magnetic potentials
in the slab and half space geometries 
has also been studied in a recent paper by Li \cite{Li1}. The half space results
in \cite{Li1} differ from the ones given in this work, by concerning the more general 
matrix valued Schrödinger equation and by assuming $C^6$ regularity on the  
magnetic potential.

The half space is perhaps the simplest example of an unbounded region with an 
unbounded boundary. It is of special interest in many applications, such as 
geophysics, ocean acoustics, and optical tomography, 
since it provides a simple model for semi infinite geometries. We would like to mention that 
the magnetic Schrödinger equation is closely related to the diffusion 
approximation of the photon transport equation, used in optical
tomography \cite{SA1}. 

The paper is organized as follows.
Section 2 contains a review of the construction of complex geometric optics  
solutions   for magnetic Schr\"odinger operators with Lipschitz continuous 
magnetic potentials. 
 In section 3 we derive the central integral identity. 
The proof of Theorem \ref{thm_2_inverse} is contained in 
sections 4 and 5. The appendix contains an extension of Green's second formula
and a statement of the unique continuation principle for easy reference.

\section{Complex geometric optics solutions}
Let $\Omega\subset \R^3$  be a bounded domain with $C^\infty$-boundary, and let 
$A\in W^{1,\infty}(\Omega,\R^3)$,  $q\in L^\infty(\Omega,\C)$.
The task of this subsection is to review the construction of  complex geometric 
optics solutions for  the magnetic Schr\"odinger equation,
\begin{equation}
\label{eq_Sch_lip}
L_{A,q}u=0\quad \textrm{in}\quad \Omega. 
\end{equation}
A complex geometric optics solution to \eqref{eq_Sch_lip} is a solution of the form
\begin{equation}
\label{eq_cgo_lip}
u(x,\zeta;h)=e^{x\cdot\zeta/h}(a(x,\zeta;h)+r(x,\zeta;h)),
\end{equation}
where $\zeta\in \C^3$, $\zeta\cdot\zeta=0$, $a$ is a smooth 
amplitude,  $r$ is a remainder, and $h>0$ is a small parameter.

In the case when $A\in C^2(\overline{\Omega})$ and $q\in L^\infty(\Omega)$, 
such solutions were constructed in \cite{DosSantos1} using the method of 
Carleman estimates, and the construction was extended to the case of less 
regular potentials in \cite{KnuSalo}, see also \cite{KLU1}.

Let $\varphi(x) =
\alpha \cdot x$, $\alpha \in \R^3$, $|\alpha|=1$. The fundamental role in the 
construction  of complex geometric optics solutions is played by the following 
Carleman estimate, 
\begin{equation}
\label{eq_Carleman_est}
h\|u \|_{H^1_{\textrm{scl}}(\Omega)}\le C\|e^{\varphi/h}h^2L_{A,q} 
e^{-\varphi/h} u\|_{L^2(\Omega)},
\end{equation}
valid for all $u\in C^\infty_0(\Omega)$ and $0<h\le h_0$, which was proved in 
\cite{DosSantos1} and \cite{KnuSalo}.  Here 
$\|u\|_{H^1_{\textrm{scl}}(\Omega)}=\|u\|_{L^2(\Omega)}+\|h\nabla 
u\|_{L^2(\Omega)}$. 

Based on the estimate \eqref{eq_Carleman_est}, the following solvability result 
was established in \cite[Proposition 4.3]{KnuSalo}.
\begin{prop}
\label{prop_solvability}
Let $A \in W^{1,\infty}(\Omega,\R^3)$, $q \in
L^{\infty}(\Omega,\C)$, $\alpha \in \R^3$, $|\alpha|=1$ and $\varphi(x) =
\alpha \cdot x$. Then there is  $C > 0$ and  $h_0 > 0 $ such that for all $h\in 
(0,h_0]$, and any $f\in L^2(\Omega)$, the equation
\[
e^{\varphi/h}h^2L_{A,q} e^{-\varphi/h} u=f\quad\textrm{in}\quad \Omega,
\]
has a solution $u\in H^1(\Omega)$ with 
\begin{align*}
\|u \|_{H^1_{\emph{\textrm{scl}}}(\Omega)}\le \frac{C}{h}\|f\|_{L^2(\Omega)}. 
\end{align*}
\end{prop}

Our basic strategy in constructing solutions of the form \eqref{eq_cgo_lip} 
is to write \eqref{eq_Sch_lip}, as
\begin{align}
\label{eq_split}
L_{\zeta} r = -L_{\zeta} a,
\end{align}
where $L_{\zeta} :=  e^{-x\cdot \zeta/h} h^2L_{A,q} e^{x\cdot \zeta/h}$. 
Then we first  search for a suitable $a$, after which we will get $r$ by
Proposition \ref{prop_solvability}. We must however take some care in choosing
$a$ and the way it depends on $h$, since we need later that 
$\| r \|_{H^1_{\emph{\textrm{scl}}}(\Omega)} \to 0$, sufficiently fast as $h \to 0$.
We need $a$ also to be smooth enough. This will be handled as in 
\cite{KnuSalo}.

We extend $A\in W^{1,\infty}(\Omega, \R^3)$ to a Lipschitz vector field, 
compactly supported  in $\tilde \Omega$, where $\tilde \Omega\subset \R^3$ is 
an open bounded set such that $\Omega\subset\subset \tilde \Omega$. We consider 
the mollification $A^\sharp := A*\psi_\epsilon\in
C_0^\infty(\tilde \Omega,\R^3)$.
Here $\epsilon>0$ is small and 
$\psi_\epsilon(x)=\epsilon^{-3}\psi(x/\epsilon)$ is the usual mollifier with 
$\psi\in C^\infty_0(\R^3)$, $0\le \psi\le 1$, and 
$\int \psi dx=1$.  We write 
$A^\flat=A-A^\sharp$. Notice that we have the following estimates for
$A^\flat$,
\begin{equation}
\label{eq_flat_est}
\|A^\flat\|_{L^\infty(\Omega)}=\mathcal{O}(\epsilon),
\end{equation}
\[
 \|\p^\alpha A^\sharp\|_{L^\infty(\Omega)}=\mathcal{O}(\epsilon^{-|\alpha|}) \quad 
\textrm{for all}\quad \alpha,
\]
as $\epsilon\to 0$. 

We shall work with a complex  $\zeta = \zeta_0+\zeta_1$ depending slightly on $h$, for which
\begin{align} \label{eq_zassum}
&\zeta\cdot \zeta = 0,\;
\zeta_0:= \alpha + i \beta,\; \alpha,\beta \in S^2,\; \alpha \cdot \beta=0, \\
&\zeta_0 \, \text{ independent of $h$ and }
\;\zeta_1=\mathcal{O}(h),\; \text{as $h\to 0$}. \nonumber
\end{align}
By expanding the conjugated operator we write the right hand
side of \eqref{eq_split} as
\begin{align}
\label{eq_La}
L_\zeta a 
=&(-h^2\Delta 
-2i(-i\zeta_0+hA)\cdot h\nabla -2\zeta_1\cdot h\nabla+h^2A^2 \nonumber \\
&-2ih\zeta_0\cdot (A^\sharp+A^\flat)-2ih\zeta_1\cdot A 
-ih^2(\nabla\cdot A)+h^2q)a. 
\end{align}
Now we want $a$ to be such that this expression decays more rapidly than
$\mathcal{O}(h)$, as $h \to 0$.

Consider the operator in \eqref{eq_La}, ignoring for the time being $a$ and its
possible dependence on $h$. We would like to eliminate from this operator
the terms that are of first order in $h$.
Notice first that $\zeta_1 = \mathcal{O}(h)$ and that we can control 
$\|A^\flat\|_{L^\infty(\Omega)}$ with $h$,
if we choose $\epsilon$ to be dependent on $h$.
Then in an attempt to eliminate first order terms in $h$, it is
natural to search for an $a$ for which
\begin{align}
\label{eq_az0eq}
\zeta_0 \cdot \nabla  a =  -i \zeta_0 \cdot A^\sharp  a, \quad \text{ in
$\Omega$}.
\end{align}

We will look for a solution of the form $a=e^\Phi$. The above equation becomes then
\begin{align} \label{eq_trans1}
\zeta_0 \cdot \nabla \Phi =  
-i \zeta_0 \cdot A^\sharp  , \quad \text{ in $\Omega$}.
\end{align}
Pick a
$\gamma \in S^2$, such that $ \gamma \bot \{  \alpha, 
\beta\}$. 

Next we consider the above equation in coordinates $y$, associated with the
basis $\{\alpha,\beta,\gamma\}$. Let $T$ be the coordinate transform
$y = Tx:=( x \cdot \alpha, x \cdot \beta, x \cdot \gamma)$.
Using the chain rule and the fact that $T^{-1} = T^*$, 
one gets that\footnote{Here $T^*$ is the transpose of $T$.}
\begin{align*}
\nabla(\Phi \circ T^{-1})(Tx) 
= T [\nabla \Phi(x)]^*.
\end{align*}
We therefore have that
\begin{align*}
(1,i,0) \cdot \nabla(\Phi \circ T^{-1})(Tx) 
&= (1,i,0) \cdot T [\nabla \Phi(x)]^* \\
&=
(\alpha \cdot\nabla+ i\beta\cdot\nabla)\Phi(x) \\
&= \zeta_0 \cdot \nabla\Phi(x).
\end{align*}
Equation \eqref{eq_trans1} gives hence the $\bar\p$-equation 
\begin{align} \label{eq_dbar}
2\p_{\bar z} \cdot (\Phi \circ T^{-1}) (y)   
 = -i \zeta_0 \cdot (A^\sharp \circ T^{-1})(y),
\end{align}
where $\p_{\bar z} = (\p_{y_1}+i\p_{y_2})/2$.
We will solve this using the Cauchy operator
\begin{align*}
N^{-1}f (x) := \frac{1}{\pi} \int_{\R^2} 
\frac{1}{s_1 + is_2} f(x - (s_1,s_2,0)) ds_1 ds_2,
\end{align*}
which is an inverse for the $\bar \p$-operator,
$N:= (\partial_{y_1} + i\partial_{y_2})/2$
(see e.g. \cite{Horm_book_2} Theorem 1.2.2).
We will need the following straightforward continuity result for the Cauchy
operator.

\newtheorem{cauchyop}[thm]{Lemma}
\begin{cauchyop} \label{cauchyop} 
Let $r>0$ and $f \in W^{k,\infty}(\R^3)$, $k\geq0$ and assume that $\supp(f) \subset
B(0,r)$. Then 
\begin{align*}
\|N^{-1}f\|_{W^{k,\infty}(\R^3)} \leq C_k\|f\|_{W^{k,\infty}(\R^3)} 
\end{align*}
for some constant $C_k>0$.
If $f \in C_0(\R^3)$, then $N^{-1}f \in C(\R^3)$.
\end{cauchyop}
\begin{proof} See e.g. \cite{MS3}.
\end{proof}

Returning to \eqref{eq_dbar} we get that
$\Phi = \frac{1}{2}N^{-1} (-i \zeta_0 \cdot (A^\sharp \circ T^{-1})) \circ T$. More explicitly  
we have
\begin{align} 
\label{eq_Phi}
\Phi(x,\zeta_0;\epsilon)
&=
\frac{1}{2\pi} \int_{\R^2} 
\frac{-i \zeta_0 \cdot A^\sharp(x-T^{-1}(s_1,s_2,0))}{s_1+is_2}ds_1ds_2,
\end{align}
where $T^{-1}(s_1,s_2,0)=s_1\alpha+s_2\beta$.
We have thus found a solution $a = e^\Phi$ to equation \eqref{eq_az0eq}.
We will choose $\epsilon$ so that it depends on $h$, which implies that $a$ will
depend on $h$.
In order to determine how the norm of $r$ will depend on $h$ and also for later
estimates, we will need to see how $\|\p^\alpha a\|_{L^\infty}$ depends on $h$. 
Lemma \ref{cauchyop} and estimate \eqref{eq_flat_est} imply the following result.

\begin{lem} Equation \eqref{eq_az0eq} has a solution $a\in
C^\infty(\overline{\Omega})$ satisfying the estimates
\begin{align}
\label{eq_ampl_est}
\|\p^\alpha a\|_{L^\infty(\Omega)}\le C_\alpha \epsilon^{-|\alpha|}\quad 
\textrm{for all}\quad \alpha. 
\end{align}
\end{lem}

We can now write the $L^\infty(\Omega)$ norm of \eqref{eq_La} as
\begin{align*}
\|L_\zeta a\|_{L^\infty(\Omega)}= 
\|-h^2 L_{A,q}a +2ih\zeta_0\cdot A^\flat a+2\zeta_1\cdot h\nabla 
a+2ih\zeta_1\cdot A a\|_{L^\infty(\Omega)}.
\end{align*}
Using \eqref{eq_flat_est}, \eqref{eq_ampl_est} and the fact that 
$\zeta_1=\mathcal{O}(h)$ we have that 
\begin{align*}
\|L_\zeta a\|_{L^\infty(\Omega)} = \mathcal{O}(h^2\epsilon^{-2}+h\epsilon).
\end{align*}
Choosing $\epsilon = h^{1/3}$, gives finally $\|L_\zeta a\|_{L^\infty(\Omega)} =
\mathcal{O}(h^{4/3})$, as $h \to 0$.

Finally to solve \eqref{eq_split} for $r$, we rewrite it as  
\begin{align} \label{eq_req}
e^{-x\cdot \Re \zeta/h} h^2L_{A,q} e^{x\cdot \Re \zeta/h}(e^{ix\cdot \Im \zeta/h}r)
= -e^{ix\cdot \Im \zeta/h}L_\zeta a.
\end{align}
If we replace $e^{ix\cdot \Im \zeta/h}r$ by $\tilde r$, then
the solvability result Proposition \ref{prop_solvability}, shows that we can find a solution $\tilde r$,
so that a solution $r$ to \eqref{eq_req} is given by $r = e^{-ix\cdot \Im \zeta/h} \tilde r$.

To get a norm estimate for $r$, notice that for the right hand side of \eqref{eq_req}
we have
\begin{align*}
\|e^{ix\cdot \Im \zeta/h}L_\zeta a\|_{L^\infty(\Omega)}  
= \mathcal{O}(h^{4/3}),
\end{align*}
as $h \to 0$. The solvability result \ref{prop_solvability} gives then that
\begin{align*}
\| \tilde r \|_{H^1_{\emph{\textrm{scl}}}(\Omega)}
= \mathcal{O}(h^{1/3}),
\end{align*}
as $h \to 0$, which implies that $\|r\|_{H^1_{\emph{\textrm{scl}}}(\Omega)}  = \mathcal{O}(h^{1/3})$,
as $h \to 0$.

Thus we have obtained the following existence result for complex geometric optics
solutions.

\begin{prop}
\label{prop_CGO_Lip}
Let $A\in W^{1,\infty}(\Omega,\R^3)$ and $q\in L^\infty(\Omega,\C)$. Then for 
$h>0$ small enough, there exist solutions $u\in H^1(\Omega)$,
of the equation 
\[
L_{A,q}u=0\quad \textrm{in}\quad \Omega, 
\]
that are of the form 
\[
u(x,\zeta;h)=e^{x\cdot\zeta/h}(a(x,\zeta;h)+r(x,\zeta;h)),
\]
where $\zeta\in \C^3$, is of the form given by
\eqref{eq_zassum}, $a\in C^\infty(\overline{\Omega})$ solves the
equation \eqref{eq_az0eq},  and where $a$ and $r$ satisfy the estimates
\[
\|\p^\alpha a\|_{L^\infty(\Omega)}\le C_\alpha h^{-|\alpha|/3} 
\quad \text{and} \quad
\|r\|_{H^1_{\emph{\textrm{scl}}}(\Omega)}=\mathcal{O}(h^{1/3}). 
\]
\end{prop}
\begin{flushright}$\Box$
\end{flushright}

\begin{rem} 
\label{rem_com_geom_1} In the sequel, we  need complex geometric optics 
solutions belonging to $H^{2}(\Omega)$. To obtain such solutions, let 
$\Omega'\supset\supset\Omega$ be a bounded domain with smooth boundary,  and 
let us extend $A\in W^{1,\infty}(\Omega,\R^3)$ and $q\in L^\infty(\Omega)$ to 
$W^{1,\infty}(\Omega',\R^3)$ and $L^\infty(\Omega')$-functions, respectively. 
By elliptic regularity, the complex geometric optics solutions, constructed on 
$\Omega'$, according to Proposition \emph{\ref{prop_CGO_Lip}},   belong to  
$H^{2}(\Omega)$.

\end{rem}

\begin{rem} \label{rem_CGO_Lip}
Recall that $\Phi = \frac{1}{2}N^{-1} (-i (\alpha+i\beta) \cdot (A^\sharp \circ T^{-1})) \circ T$. Lemma \ref{cauchyop} implies that
$N^{-1}:C_0(\Omega)\to C(\Omega)$ is continuous. The estimates \eqref{eq_flat_est}
show that $A^\sharp \to A$ uniformly on $\Omega$. It follows that, if we define 
$\Phi^0:=\frac{1}{2}N^{-1} (-i (\alpha+i\beta) \cdot (A \circ T^{-1})) \circ T$, 
then $\Phi^0$ solves the equation
\begin{align}
\label{eq_Phi0}
\zeta_0 \cdot \nabla \Phi^0 =-i \zeta_0 \cdot A \quad \textrm{in}\quad \Omega, 
\end{align}
and satisfies
\[
\|\Phi(x,\zeta_0;h^{1/3})- \Phi^0 \|_{L^\infty(\Omega)}\to 0, \quad h\to 0.
\]
\end{rem}

\begin{rem}  \label{rem_g}
We shall later use a slightly more general form for the amplitude $a$ 
in the complex geometric optics solutions. Namely we suppose that $a =
g e^{\Phi}$, where $g \in C^{\infty}(\ov{\Omega})$, is such that
\begin{align}
\zeta_0 \cdot \nabla g =0. \label{eq:gcond}
\end{align}
This means that $g$ is holomorphic in a plane spanned by $\alpha$ and
$\beta$.
Notice also that by picking $a=ge^\Phi$, we get by \eqref{eq_az0eq} that
\begin{align*}
 \zeta_0 \cdot g\nabla \Phi = -i \zeta_0 \cdot gA^\sharp,
\end{align*}
in place of \eqref{eq_trans1}. But the $\Phi$ solving \eqref{eq_trans1}
also solves the above equation. Hence we can use the same argument to obtain the $\Phi$
for the above equation, as earlier.

We thus obtain CGO solutions of the form
\begin{align*}
u = e^{x\cdot \zeta/h}(ge^{\Phi} + r_g),
\end{align*}
where $\Phi$ solves \eqref{eq_trans1}.

Notice also that setting $a=ge^{\Phi}$ does not affect the norm estimates on $a$
in Proposition \ref{prop_CGO_Lip}, since $g$ does not depend on $h$.
\end{rem}

\section{An integral identity}
\label{sec:IA}

One central step in the ideas that are used in proving uniqueness results for 
inverse boundary value problems, is to derive an integral equation that
expresses $L^2$
orthogonality between the product of two solutions $u_1$ and $u_2$,  and  the 
difference of two
potentials $q_1$ and $q_2$, see \cite{Uhl_review_1999}. One shows that
\begin{align*}
\int (q_1-q_2)u_1 u_2=0,
\end{align*}
provided that the DN-maps for $q_1$ and $q_2$ are equal.

A similar thing will be done in this subsection, for the magnetic case. The
integral equation, is however more involved in the case of a magnetic potential 
and will not by
itself be interpreted as an orthogonality relation. 
We will be considering the integral equation in conjunction with solutions that
depend on a small positive parameter $h$. In the later sections we will see 
that 
in the limit $h \nuoli 0$, we obtain a criterion for the curl being zero. 

It will be convenient to set
\[
    l:=\p \HS \cap \ov{N}, 
\]
Recall that we assume that 
\[
(\p\HS  \setminus \ov{N}) \cap \Gamma_j \neq \emptyset, \quad j=1,2.
\]
We can thus choose $\tilde \Gamma_j$, such that 
\begin{align*}
\tilde \Gamma_j\subset\Gamma_j,\quad \tilde \Gamma_j \subset \subset 
\p\HS \setminus \ov{N} ,\quad j=1,2. 
\end{align*}
Then it follows from \eqref{eq_data_inv} that 
\begin{equation}
\label{eq_data_inv_1}
\Lambda_{A_1,q_1}(f)|_{\tilde \Gamma_1}=\Lambda_{A_2,q_2}(f)|_{\tilde \Gamma_1},
\end{equation}
for any $f\in H^{3/2}(\p \HS)$, $\supp(f)\subset \tilde \Gamma_2$.  In order 
to prove Theorem \ref{thm_2_inverse} we shall only  use the data 
\eqref{eq_data_inv_1}, which turns out to be enough to determine the magnetic 
field and the electric potential. 

We now begin deriving the integral identity. We 
assume that $A_j,q_j$ and $\Gamma_j$ 
are as in Theorem \ref{thm_2_inverse} so that \eqref{eq_data_inv_1} also applies.

Let $u_1\in 
H^2_{\textrm{loc}}(\overline{\HS})$ be the radiating solution to 
\begin{align*}
(L_{A_1,q_1} -k^2) u_1 &= 0 \text{ in $\HS$},  \\
u_1|_{\partial \HS} &= f, 
\end{align*}
with $f \in H^{3/2}(\p \HS)$, $\supp(f)\subset\tilde \Gamma_2$. Let $v\in 
H^2_{\textrm{loc}}(\overline{\HS})$ be the radiating solution to 
\begin{align*}
(L_{A_2,q_2} -k^2) v &= 0 \text{ in $\HS$}, \\
v|_{\partial \HS} &= f.
\end{align*}
Define $w:=v-u_1$. Then
\begin{align}
(L_{A_2,q_2} -k^2) w
&= 2i(A_2-A_1)\cdot \nabla u_1 + i \nabla \cdot(A_2-A_1)u_1 \nonumber \\
& \quad +(A_1^2-A_2^2)u_1 + (q_1-q_2)u_1. \label{eq:aoI}
\end{align}
It follows from \eqref{eq_data_inv_1} that
\begin{align}
\label{eq_sec3_1}
(\partial_n + i A_1\cdot n) u_1|_{\tilde \Gamma_1} = (\partial_n + i A_2\cdot 
n)v|_{\tilde \Gamma_1}. 
\end{align}

By Lemma \ref{gauge_inv} we may and shall assume that 
$A_1\cdot n=A_2\cdot n=0$ on $\p\HS$, so that $\p_n w=0$ on 
$\tilde \Gamma_1$.  
We also conclude from \eqref{eq:aoI} that $w$ satisfies the Helmholtz equation
\[
(-\Delta-k^2)w=0\quad \textrm{in}\quad \HS\setminus\ov{N}. 
\]
As $w|_{\tilde \Gamma_1}=\p_n w|_{\tilde \Gamma_1}=0$, by unique continuation, we get that 
$w=0$ in $\HS \setminus \ov{N}$
(see Theorem \ref{UCP} and 
Corollary \ref{UCP_boundary} in the appendix).
Since $w\in H^2_{\textrm{loc}}(\overline{\HS})$, we have 
\[
w=\p_n w=0\quad\textrm{on} \quad \p N \cap \HS.
\]
Let $u_2\in H^2( N)$ be a solution to $(L_{A_2,\overline{q_2}}-k^2)u_2 = 0$ 
in $N$.
Then by Green's formula, we get 
\begin{align*}
((L_{A_2,q_2}-k^2) w, u_2)_{L^2(N)} 
&= (w , (L_{A_2,\overline{q_2}}-k^2) u_2)_{L^2(N)} \\
&\quad - ((\partial_n + iA_2\cdot n) w, u_2)_{L^2(\partial N)} \\
&\quad + (w , (\partial_n + iA_2\cdot n) u_2)_{L^2(\partial N)} \\
&= -(\partial_n w, u_2)_{L^2(l)}.
\end{align*}
Assuming that 
\[
u_2=0\quad \textrm{on}\quad l,
\]
we conclude that 
\[
((L_{A_2,q_2}-k^2) w, u_2)_{L^2(N)}=0. 
\]
Using equation \eqref{eq:aoI} we may write this as follows,  
\begin{align*}
\int_{N} (2i(A_2-A_1)\cdot (\nabla u_1)\ov{u_2} + i \nabla \cdot(A_2-A_1)u_1 
\ov{u_2})\,dx\\
+ \int_{N}(A_1^2-A_2^2 + q_1-q_2)u_1\ov{u_2}\,dx= 0.
\end{align*}
Using again the fact that $(A_2-A_1)\cdot n=0$ on $\p N$ and  an integration by 
parts, we get 
\begin{align*}
i \int_{N} \nabla \cdot(A_2-A_1)u_1 \ov{u_2}\,dx
= -i \int_{N} (A_2-A_1) \cdot (\nabla u_1 \overline{u_2} 
+ u_1 \nabla \overline{u_2})dx.
\end{align*}
Thus, we obtain that  
\begin{equation}
\label{eq_sec3_3}
\begin{aligned}
\int_{N}& i(A_2-A_1) \cdot (\nabla u_1 \ov{u_2} - u_1 \nabla 
\ov{u_2})\,dx\\ 
 &+
\int_{N} (A_1^2-A_2^2 + q_1-q_2)u_1\ov{u_2}\,dx = 0, 
\end{aligned}
\end{equation}
where $u_1\in W_1(\HS)$ and $u_2\in W_2^*(N)$. Here
\begin{align*}
W_1(\HS):=\{u\in H^2_{\textrm{loc}}(\overline{\HS})\;:\;  
(L_{A_1,q_1}-k^2)u=0\textrm{ in }\HS, \\
\supp(u_1|_{\p\HS})\subset\tilde \Gamma_2, u\textrm{ radiating}\},
\end{align*}
and
\begin{align*}
W_2^*(N):=\{u\in H^2(N)\; : \; (L_{A_2,\ov{q_2}}-k^2)u=0\textrm{ in } 
 N, u|_{l}=0\}.
\end{align*}

We shall next extend the integral identity \eqref{eq_sec3_3} to a richer class 
of solutions to the magnetic Schr\"odinger equation. To that end, let us 
introduce the following space of solutions,
\[
W_1(N):=\{u\in H^2(N)\; : \; (L_{A_1,q_1}-k^2)u=0\textrm{ in }  N, 
u|_{l}=0\}.
\] 
The following Runge type approximation result is similar to those found in 
\cite{I1}, \cite{LU1} and \cite{KLU1}.

\begin{lem} \label{runge}

The space $V:=W_1(\HS)|_{N}$ is dense in $W_1(N)$ in the 
$L^2(N)$--topology.

\end{lem}

\begin{proof}

Suppose that $V$ is not dense in $W_1(N)$.
First notice that $\text{span}(V)=V$ 
so that $\ov{V}$ is a linear subspace of $L^2(N)$. Since
$V$ is not dense in $W_1(N)$, we have a vector  $u_0 \in W_1( N)$ 
such that $u_0 \notin \ov{V}$.
We can decompose  $u_0$ as $u_0=a+b$, where $a \in
\ov{V}$, $b \in \ov{V}^\bot$ and $b \neq 0$.
Let $T$ be the linear functional on $L^2( N)$, defined by 
$T(x):= \text{proj}_{\ov{V}^\bot}(x)/ \|b\|_{L^2}$, 
where $\text{proj}_{\ov{V}^\bot}$ is the orthogonal
projection to $\ov{V}^\bot$. Now clearly
$\|T(u_0)\|_{L^2} = 1$ and $T|_{V} = 0$. 

By the Riesz representation theorem, there is  $g_T \in L^2(N)$ that
corresponds to $T$. Extend $g_T$ by zero to the complement of $N$ 
in $\HS$. 
Let $U\in H^2_{\textrm{loc}}(\overline{\HS})$ be the incoming solution to 
\begin{align*}
(L_{A_1,\ov{q_1}} -k^2) U &= g_T \quad \text{in}\quad \HS,\\
U|_{\partial \HS} &= 0. 
\end{align*}
The existence of such a solution is proved in \cite{Poh1}.

Now let $u \in W_1(\HS)$. Then because $T|_{V}=0$ and $\supp(g_T) \subset 
 N$, we get 
by the Green's formula of Lemma \ref{MagGFII} that
\begin{align*}
0=(u,g_T)_{L^2(\HS)}
&= (u,(L_{A_1,\ov{q_1}}-k^2) U)_{L^2(\HS)} \\
&= ((L_{A_1,q_1}-k^2) u,U)_{L^2(\HS)} \\
&\quad - (u , (\partial_n +iA_1\cdot n) U)_{L^2(\partial \HS)} \\
&\quad + ( (\partial_n + iA_1\cdot n) u, U)_{L^2(\partial \HS)} \\
&= -( u,\partial_n U)_{L^2(\tilde \Gamma_2)}.
\end{align*}
Since the boundary condition $u|_{\tilde \Gamma_2}$ can be chosen  arbitrarily 
from
$C^\infty_0(\tilde \Gamma_2)$,
we get  that $\partial_n U|_{\tilde \Gamma_2} = 0$. Since $ U|_{\tilde 
\Gamma_2} = 0$, we 
apply the unique continuation
principle 
to conclude that  $U|_{\HS \setminus \ov{ N}} = 0$. As $U \in 
H^2_{\textrm{loc}}(\overline{\HS})$, we have 
\[
U|_{\p N \cap \HS}= \p_n U|_{\p N\cap \HS}= 0.
\]

Now applying Green's formula and 
doing the same computation as above for $u_0$ and $ N$ instead of $u$ yields
\begin{align*}
(u_0,g_T)_{L^2( N)}
&= (u_0,(L_{A_1,\ov{q_1}}-k^2) U)_{L^2( N)} \\
&= ((L_{A_1,q_1}-k^2) u_0,U)_{L^2( N)} \\
&\quad - (u_0 , (\partial_n +iA_1\cdot n) U)_{L^2(\partial  N)} \\
&\quad + ( (\partial_n + iA_1\cdot n) u_0, U)_{L^2(\partial  N)} \\
&=-(u_0 , \partial_n U)_{L^2( l)}=0.
\end{align*}
Here we have used that $u_0|_{l}=0$. 
It follows that $T(u_0) = 0$. This contradiction completes the proof.
\end{proof}

Since $(A_2-A_1)\cdot n=0$ on $\p N$,  we can rewrite \eqref{eq_sec3_3} in 
the following form,
\begin{align*}
-\int_{ N}& u_1 i \nabla \cdot((A_2-A_1) \ov{u_2})\,dx  -\int_{ N} i   
(A_2-A_1)\cdot (u_1 \nabla \ov{u_2})\,dx\\ 
 &+
\int_{ N} (A_1^2-A_2^2 + q_1-q_2)u_1\ov{u_2}\,dx = 0.
\end{align*}
Hence, an application of Lemma \ref{runge} implies that the integral identity  
\eqref{eq_sec3_3} is valid for any $u_1\in W_1( N)$ and any $u_2 \in 
W_2^*( N)$.

We summarize the discussion in this subsection in the following result.

\begin{prop} \label{identity} 
Assume that $A_j,q_j$ and $\Gamma_j$, $j=1,2$ 
are as in Theorem \ref{thm_2_inverse} and that the DN-maps satisfy
\begin{align}
\label{eq_ieDN}
\Lambda_{A_1,q_1}(f)|_{\Gamma_1}=\Lambda_{A_2,q_2}(f)|_{\Gamma_1},
\end{align}
for any $f\in H^{3/2}_{\emph{\textrm{comp}}}(\p \HS)$, $\supp(f)\subset \Gamma_2$.
Then 
\begin{equation}
\begin{aligned}\label{eq_sec3_4} 
\int_{ N}& i(A_2-A_1) \cdot (\nabla u_1 \ov{u_2} - u_1 \nabla 
\ov{u_2})\,dx\\ 
 &+
\int_{ N} (A_1^2-A_2^2 + q_1-q_2)u_1\ov{u_2}\,dx = 0, 
\end{aligned}
\end{equation}
for any $u_1\in W_1( N)$ and any $u_2 \in W_2^*( N)$.
\end{prop}
\begin{flushright}$\Box$
\end{flushright}

\begin{rem} \label{rem_ie}
Notice that the proof of Proposition \ref{identity} only uses the assumption
\eqref{eq_data_inv_1}, which follows from \eqref{eq_ieDN}. 
Proposition \ref{identity} holds therefore also under the weaker assumption 
\eqref{eq_data_inv_1}.
\end{rem}

\section{Recovering the magnetic field}

The aim of this section is to prove the first part of Theorem \ref{thm_2_inverse}, by
showing that the curl of the magnetic potential is
determined by the DN-map.
We choose an open ball $B$ centered on $\p \HS$ with $N \subset \subset B$.
And use the notations
\[
\Bm:=\HS\cap B,\quad B_+:=\R^3_+\cap B ,\quad l_B:=\p\HS\cap B.
\]
The first step in the argument will be to 
construct complex geometric optics solutions $u_1 \in W_1(N)$ and $u_2  \in W_2^*(N)$ 
and then to examine the limit of \eqref{eq_sec3_4} as $h \to 0$.

For $u_1\in W_1(N)$ and $u_2\in W_2^*(N)$, we have that 
$u_j|_{l}=0$, $j=1,2$. To obtain solutions that satisfy this condition, we will
first choose solutions defined on the larger set $B$ and then use a reflection
argument.

The parameters $\zeta$ for the complex geometric optics
solutions will be picked as follows. We will assume that   
\begin{align}
\label{eq_gamma-assum}
\xi,\gamma_1,\gamma_2\in\R^3,\; |\gamma_1|=|\gamma_2|=1\; 
\text{ and that }\{\gamma_1,\gamma_2,\xi\} \; \text{is orthogonal}. 
\end{align}
Similarly to \cite{S1}, we set 
\begin{align}
\label{eq_zeta_1_2}
\zeta_1&=\frac{ih\xi}{2}+\gamma_1+ i\sqrt{1-h^2\frac{|\xi|^2}{4}}\gamma_2, \\ 
\zeta_2&=-\frac{ih\xi}{2}-\gamma_1+i\sqrt{1-h^2\frac{|\xi|^2}{4}}\gamma_2, \nonumber 
\end{align}
so that $\zeta_j\cdot\zeta_j=0$, $j=1,2$, and 
$(\zeta_1+\overline{\zeta_2})/h=i\xi$. Here $h>0$ is a small 
semiclassical parameter. Notice also that $\zeta_j$, $j=1,2$ satisfy the
conditions on $\zeta$ in \eqref{eq_zassum}, when we take 
$\alpha=\pm\gamma_1$ and $\beta=\gamma_2$.

We need to extend the potentials $A_j$ and $q_j$, $j=1,2$, to  $B_+$.  For
the component functions 
$A_{j,1}$, $A_{j,2}$, and $q_j$, we do an even extension, and  for $A_{j,3}$, we 
do an odd extension, i.e.,  for $j=1,2$ we set,
\begin{align*}
\tilde A_{j,k}(x)&=\begin{cases} A_{j,k}(x),& x_3<0,\\
 A_{j,k}(\tilde x),& x_3>0,
\end{cases},\quad k=1,2,\\
\tilde A_{j,3}(x)&=\begin{cases} A_{j,3}(x),& x_3<0,\\
- A_{j,3}(\tilde x),& x_3>0,
\end{cases}\\
\tilde q_j(x)&=\begin{cases} q_j(x),& x_3<0,\\
 q_j(\tilde x),& x_3>0,
\end{cases}
\end{align*}  
where $\tilde x := (x_1,x_2,-x_3)$.
By our assumptions, $A_{j,3}|_{x_3=0}=0$, from which it follows  
that  $\tilde A_j\in W^{1,\infty}(B)$ and 
$\tilde q_j\in L^{\infty}(B)$, $j=1,2$.

We can now by Proposition \ref{prop_CGO_Lip} and Remark \ref{rem_com_geom_1}
pick complex geometric optics solutions  $\tilde u_1$ in $H^2(B)$,
\[
 \tilde u_1(x,\zeta_1;h)=e^{x\cdot \zeta_1/h} 
(e^{\Phi_1(x,\gamma_1+i\gamma_2;h)}+r_1(x,\zeta_1; h))
\]
of the equation $(L_{\tilde A_1, \tilde q_1}-k^2) \tilde u_1=0$ in $B$, 
where $\Phi_1\in C^{\infty}(\overline{B})$.
By Remark \ref{rem_CGO_Lip}, $\Phi_1 \to \Phi_1^0$ in 
the $L^\infty$-norm as $h\to 0$, where 
$\Phi_1^0$ solves the equation
\begin{align}
\label{eq_Phi10}
(\gamma_1+i\gamma_2) \cdot \nabla \Phi_1^0 = -i(\gamma_1+i\gamma_2) \cdot \tilde A_1 \quad\text{in}\quad B.
\end{align}
To obtain a function that is zero on the plane $x_3=0$, we set
\begin{align}
\label{eq_u_1-cgo}
u_1(x) := \tilde u_1(x)-\tilde u_1(\tilde x),\quad x\in \Bm \cup l_B. 
\end{align}
Then it is easy to check that the restriction $u_1|_N\in W_1(N)$.

We can similarly pick by Proposition \ref{prop_CGO_Lip} and Remark \ref{rem_com_geom_1},
complex geometric optics solutions $\tilde u_2$ in $H^2(B)$,
\[
\tilde u_2(x,\zeta_2;h)=e^{x\cdot \zeta_2/h} 
(e^{\Phi_2(x,-\gamma_1+i\gamma_2;h)}+r_2(x,\zeta_2; h))
\]
of the equation $(L_{\tilde A_2,\overline{\tilde q_2}}-k^2)\tilde u_2=0$ in 
$B$, 
where $\Phi_2\in C^{\infty}(\overline{B})$.
By Remark \ref{rem_CGO_Lip}, $\Phi_2 \to \Phi_2^0$ in 
the $L^\infty$-norm as $h\to 0$, where 
$\Phi_2^0$ solves the equation
\begin{align}
\label{eq_Phi20}
(-\gamma_1+i\gamma_2) \cdot \nabla \Phi_2^0 = -i(-\gamma_1+i\gamma_2) \cdot \tilde A_2 \quad\text{in}\quad B.
\end{align}
To obtain a function that is zero on the plane $x_3=0$, we set
\begin{align}
\label{eq_u_2-cgo}
u_2(x) := \tilde u_2(x)-\tilde u_2(\tilde x),\quad x\in \Bm \cup l_B. 
\end{align}
Then it is easy to check that the restriction $u_2|_N\in W_2^*(N)$.

The next step is to substitute the complex geometric optics solutions $u_1$ and 
$u_2$, given by \eqref{eq_u_1-cgo} and \eqref{eq_u_2-cgo}, respectively,  into 
the integral identity \eqref{eq_sec3_4}. This  will be done in the Lemma bellow.
We will use the abbreviations $P_1(x) := e^{\Phi_1(x)}+r_1(x)$ and $P_2(x) :=
e^{\Phi_2(x)}+r_2(x)$, so that
\begin{align*}
u_1(x)&=e^{x\cdot\zeta_1/h}P_1(x) -e^{\tilde x\cdot\zeta_1/h}P_1(\tilde x), \\
u_2(x)&=e^{x\cdot\zeta_2/h}P_2(x) -e^{\tilde x\cdot\zeta_2/h}P_2(\tilde x).
\end{align*}
For future references, it 
will be convenient to compute the product of the phases that occur in the terms 
$u_1\ov{u}_2, \nabla u_1\ov{u_2}$ and $u_1 \nabla \ov{u_2}$
\begin{equation}
\label{eq_phases}
\begin{aligned}
e^{x\cdot\zeta_1/h}e^{x\cdot\overline{\zeta_2}/h}&=e^{ix\cdot\xi},\quad 
e^{\tilde x\cdot\zeta_1/h}e^{\tilde x\cdot\overline{\zeta_2}/h}=e^{i\tilde x 
\cdot\xi},\\
e^{\tilde x 
\cdot\zeta_1/h}e^{x\cdot\overline{\zeta_2}/h}
&=e^{ix\cdot\xi}e^{i(0,0,-2x_3)\cdot\zeta_1/h}
=e^{ix\cdot  \xi_--2\gamma_{1,3} x_3/h},\\
e^{x\cdot\zeta_1/h}e^{\tilde x\cdot\overline{\zeta_2}/h}
&=e^{i\tilde x\cdot\xi}e^{i(0,0,2x_3)\cdot\zeta_1/h} 
=e^{ix\cdot \xi_++2\gamma_{1,3} x_3/h},
\end{aligned}
\end{equation}
where $\gamma_j=(\gamma_{j,1},\gamma_{j,2},\gamma_{j,3})$, $j=1,2$ and 
\[
\xi_\pm=\bigg(\xi_1,\xi_2,\pm 
\frac{2}{h}\sqrt{1-\frac{h^2|\xi|^2}{4}}\gamma_{2,3}\bigg).
\]
We restrict the choices of $\gamma_1$ and $\gamma_2$, by assuming that
\begin{align} \label{eq_gamma-rest}
\gamma_{1,3}= 0 \quad \text{and} \quad  \gamma_{2,3}\ne 0.
\end{align}
We need these conditions for the proof of the next Lemma. 
The first condition makes the above phases
purely imaginary, which avoids exponential growth of the terms, as $h\to0$ . 
The second condition implies that $|\xi_\pm|\to \infty$ as $h\to 0$.
This will be needed since we will use the Riemann-Lebesgue Lemma 
to eliminate unwanted imaginary exponentials. 

Finally it will also be convenient to explicitly state the following norm estimates, 
which follow from Proposition \ref{prop_CGO_Lip} 
\begin{equation}
\label{eq_rem_amp}
\begin{aligned}
&\|e^{\Phi_j}\|_{L^\infty}=\mathcal{O}(1),\quad \|\nabla 
e^{\Phi_j}\|_{L^\infty}=\mathcal{O}(h^{-1/3}),\\
&\|r_j\|_{L^2}=\mathcal{O}(h^{1/3}),\quad \|\nabla 
r_j\|_{L^2}=\mathcal{O}(h^{-2/3}),\quad j=1,2,
\end{aligned}
\end{equation}
as $h\to0$.

\begin{lem}
\label{inter_prop_32}
We have
\begin{equation}
\label{eq_identity_main_2}
(\gamma_1+i\gamma_2)\cdot\int_{B}(\tilde A_2-\tilde 
A_1)e^{ix\cdot\xi}e^{\Phi_1^0+\ov{\Phi_2^0}}dx=0,
\end{equation}
where $\gamma_1,\gamma_2$ and $\xi$ satisfy \eqref{eq_gamma-assum} and \eqref{eq_gamma-rest}.
\end{lem}
\begin{proof} We will prove the statement by
multiplying the integral equation \eqref{eq_sec3_4} of Proposition \ref{identity} by $h$,
when $u_1$ and $u_2$ are given by \eqref{eq_u_1-cgo} and \eqref{eq_u_2-cgo}, and
then take the limit as $h \to 0$.

To begin with notice that we may integrate over $\Bm$ in \eqref{eq_sec3_4},
since 
\[
    \supp(A_j),\supp(q_j) \subset N \subset \ov{\Bm}
\]
and $u_j$ are defined in $B$, when $j=1,2$.
We first show that for the second term in \eqref{eq_sec3_4} we have 
\begin{align}\label{eq_ua1}
h \int_{\Bm} (A_1^2-A_2^2 + q_1-q_2)u_1\ov{u}_2 dx \to 0,
\end{align}
as $h \to 0$.
Using the phase computations \eqref{eq_phases} we get that
\begin{align*}
u_1 \ov{u}_2 = 
&\quad e^{i\xi \cdot x} P_1(x) \ov{P}_2(x) 
- e^{ix \cdot \xi_+} P_1(x) \ov{P}_2(\tilde x) \\
&- e^{ix \cdot \xi_-}P_1(\tilde x) \ov{P}_2(x) 
+ e^{i\xi \cdot \tilde x} P_1(\tilde x) \ov{P}_2(\tilde x).
\end{align*}
This is multiplied by an $L^\infty$ function in \eqref{eq_ua1}. Since we
restricted the choice of $\gamma_1$ to make the exponents purely imaginary, we
see easily using the estimates \eqref{eq_rem_amp} that \eqref{eq_ua1} holds.

Equation \eqref{eq_sec3_4} multiplied by $h$, is thus reduced in the limit to
\begin{align}
\label{eq_ua2}
\lim_{h \to 0} \bigg(
h \int_{\Bm} i(A_2-A_1)\cdot\nabla u_1 \ov{u}_2 dx 
- h \int_{\Bm} i(A_2-A_1)\cdot u_1 \nabla \ov{u}_2  dx \bigg)=0.
\end{align}
We will proceed by examining the first term.
Using \eqref{eq_phases} we  write $\nabla u_1 \ov{u}_2$ as
\begin{align*}
\nabla u_1 \ov{u}_2 = 
&\frac{\zeta_1}{h} \big( 
e^{ix \cdot \xi} P_1(x) \ov{P_2(x)}-  e^{ix \cdot \xi_+} P_1(x) \ov{P_2(\tilde x)} \big)  \\
&+e^{ix \cdot \xi} \nabla P_1(x) \ov{P_2(x)}-  e^{ix \cdot \xi_+} \nabla P_1(x) \ov{P_2(\tilde x)} \\
&-\frac{\tilde \zeta_1}{h} \big( 
e^{ix \cdot \xi_-}  P_1(\tilde x) \ov{P_2(x)} - e^{i\tilde x \cdot \xi} P_1(\tilde x) 
\ov{P_2(\tilde x)} \big)  \\
&-e^{ix \cdot \xi_-} \nabla_x P_1(\tilde x) \ov{P_2(x)} + e^{i\tilde x \cdot
\xi} \nabla_x P_1( \tilde x) 
\ov{P_2(\tilde x)},
\end{align*}
where $\tilde \zeta_j := (\zeta_{j,1},\zeta_{j,2},-\zeta_{j,3})$, $j=1,2$.
The terms of the product that do not contain the factor $1/h$, result in
integrals similar
to the one in \eqref{eq_ua1}. One sees similarly using estimates  \eqref{eq_rem_amp} that they are 
zero in the limit of \eqref{eq_ua2}. 
The first term inside the limit in \eqref{eq_ua2} is therefore reduced to 
\begin{align*}
\lim_{h \to 0} 
\int_{\Bm} i(A_2-A_1)\cdot \big( 
&\zeta_1 e^{ix \cdot \xi} P_1(x) \ov{P_2(x)}   - \tilde \zeta_1 e^{ix \cdot \xi_-} P_1(\tilde x) \ov{P_2(x)} \\
-&\zeta_1  e^{ix \cdot \xi_+} P_1(x)\ov{P_2(\tilde x)}  + \tilde \zeta_1
e^{i\tilde x \cdot \xi} P_1(\tilde x) \ov{P_2(\tilde x)}  \big) dx.
\end{align*}
Now we use the Riemann-Lebesgue Lemma to conclude that the terms with exponents
containing  $\xi_+$ and $\xi_-$ are zero in the limit. 
To see this, notice that by Remark \ref{rem_CGO_Lip}, we see
that $\|\Phi_i\|_{L^\infty(\Bm)} < C$, for some $C>0$, when $h$ is small enough. Estimates
\eqref{eq_rem_amp} show that
$\|r_i\|_{L^1(\Bm)} = \mathcal{O}(h^{1/3})$. Hence
$\|P_iP_j\|_{L_1(\Bm)}<C$, for some $C>0$ when $h$ is small enough.
Finally we have $\xi_\pm \to \infty$, as $h\to0$, because of
the restrictions \eqref{eq_gamma-rest}.

The first term in \eqref{eq_ua2} is therefore 
\begin{align*}
\lim_{h \to 0} 
\int_{\Bm} i(A_2-A_1)\cdot \big( 
\zeta_1 e^{ix \cdot \xi} P_1(x) \ov{P_2(x)} 
+ \tilde \zeta_1 e^{i\tilde x \cdot \xi} P_1(\tilde x) \ov{P_2(\tilde x)} 
\big) dx
\end{align*}
as $h\to 0$. The terms containing $r_i$ in the products of $P_1$ and $P_2$
are, because of \eqref{eq_rem_amp}, zero in the limit. 
The above limit is thus equal to
\begin{align*}
\lim_{h \to 0} 
\int_{\Bm} i(A_2-A_1)\cdot \big( 
\zeta_1 e^{ix \cdot \xi} 
e^{\Phi_1(x) + \ov{\Phi_2(x)}}  
+ \tilde \zeta_1 e^{i\tilde x \cdot \xi} 
e^{\Phi_1(\tilde x) + \ov{\Phi_2(\tilde x)}} \big) dx.
\end{align*}
Finally we split the integral and  do a change of variable in the second term
and  arrive at the expression
\begin{align}
\label{eq_ua3}
\lim_{h \to 0} \quad 
\int_{B} i(\tilde A_2- \tilde A_1)  \cdot \zeta_1 
e^{ix \cdot \xi} 
e^{\Phi_1(x) + \ov{\Phi_2(x)}} dx,  
\end{align}
for the first term of \eqref{eq_ua2}.

Returning to the second term in \eqref{eq_ua2}, containing $u_1 \nabla
\ov{u_2}$. This is of the
same form as the first one. By doing the above derivation by simply exchanging the
roles of  $u_1$ and $\ov{u_2}$, we similarly see that the second term becomes
\begin{align}
\label{eq_ua4}
\lim_{h \to 0} \quad 
-\int_{B} i(\tilde A_2- \tilde A_1)
\cdot \ov{\zeta_2} 
e^{ix \cdot \xi} 
e^{\Phi_1(x) + \ov{\Phi_2(x)}} dx.  
\end{align}
Now $\zeta_1 \to (\gamma_1 + i \gamma_2)$ and $\ov{\zeta_2} \to - (\gamma_1+i\gamma_2)$,
as $h\to0$.
Thus by using \eqref{eq_ua3} with \eqref{eq_ua4}, we can rewrite \eqref{eq_ua2} as
\begin{align*}
\lim_{h \to 0} &\quad 
\int_{B} i(\tilde A_2-\tilde A_1)  \cdot \big( \zeta_1 
e^{ix \cdot \xi} 
e^{\Phi_1(x) + \ov{\Phi_2(x)}} 
-\ov{\zeta_2} 
e^{ix \cdot \xi} 
e^{\Phi_1(x) + \ov{\Phi_2(x)}}\big) dx \\
&=
2\int_{B} i(\tilde A_2-\tilde A_1)  \cdot (\gamma_1 + i \gamma_2)
e^{ix \cdot \xi} 
e^{\Phi_1^0(x) + \ov{\Phi_2^0(x)}} dx = 0. 
\end{align*}
\end{proof}

The next Proposition shows that \eqref{eq_identity_main_2} holds even 
when the exponential function depending on $\Phi_i^0$, $i=1,2$ is removed.
The argument follows \cite{DosSantos1} closely. We will give details for the convenience
of the reader.

\begin{prop}
\label{prop_32}
The equality \eqref{eq_identity_main_2} implies that 
\begin{equation}
\label{eq_identity_main_2_new}
(\gamma_1 + i \gamma_2) \cdot \int_{B} (\tilde A_2-\tilde 
A_1)e^{ix\cdot\xi}dx=0,
\end{equation}
for $\gamma_1,\gamma_2$ and $\xi$ which satisfy \eqref{eq_gamma-assum} and \eqref{eq_gamma-rest}.
\end{prop}

\begin{proof}
By \eqref{eq_Phi10} and \eqref{eq_Phi20} we have that
\begin{equation}
\label{eq_amplitude_sum}
(\gamma_1 + i\gamma_2)\cdot\nabla 
(\Phi_1^0+\overline{\Phi_2^0})=- i(\gamma_1+i\gamma_2)\cdot (\tilde 
A_1-\tilde A_2) \quad \textrm{in}\quad B. 
\end{equation}
Remark \ref{rem_g} furthermore implies that  the amplitude $e^{\Phi_1}$ in 
the definition of $u_1$ can be replaced  by $ge^{\Phi_1}$, 
if $g\in C^\infty(\overline{B})$ is a solution of 
\begin{equation}
\label{eq_g}
(\gamma_1+ i \gamma_2)\cdot \nabla g=0\quad  \textrm{in}\quad B.  
\end{equation}
Let $\Psi(x) :=\Phi_1^0(x)+\overline{\Phi_2^0}(x)$.
Then instead of \eqref{eq_identity_main_2} we can write, 
\[
(\gamma_1+ i\gamma_2)\cdot \int_{B} (\tilde A_2 -\tilde 
A_1)ge^{ix\cdot\xi}e^{\Psi(x)}dx=0.
\]
We conclude from \eqref{eq_amplitude_sum} that
\[
(\gamma_1 + i\gamma_2)\cdot (\tilde A_2 - \tilde A_1)ge^{\Psi} 
=-i (\gamma_1 + i\gamma_2 )\cdot(g\nabla e^{\Psi}),
\]
and therefore, we get
\begin{equation}
\label{eq_identity_main_3}
\int_{B} ge^{ix\cdot \xi}  (\gamma_1 + i\gamma_2)\cdot\nabla 
e^{\Psi} dx=0,
\end{equation}
for all $g$ satisfying \eqref{eq_g}. 

We pick a $\gamma_3$, with $|\gamma_3|=1$, so that we obtain an orthonormal basis
$\{\gamma_1,\gamma_2, \gamma_3\}$. Let $T$ be the coordinate transform into this
basis, i.e.
$y = Tx = (x\cdot\gamma_1, x\cdot\gamma_2, x\cdot\gamma_3)$. 
Set $z=y_1+i y_2$, so that $\p_{\bar z}=(\p_{y_1}+i\p_{y_2})/2$ and 
\[
(\gamma_1 + i\gamma_2)\cdot\nabla =2\p_{\bar z}. 
\]
Rewriting
\eqref{eq_identity_main_3} using this and a change of variable given by $T$ we have
\begin{align*}
\int_{TB} ge^{iy\cdot \xi}  \p_{\ov{z}} e^{\Psi} dy = 0,
\end{align*}
for all $g$ satisfying \eqref{eq_g}. 

Notice that $y\cdot \xi = y_3\xi_3$, since $\xi$ is in the $y$-coordinates of the
form $(0,0,\xi_3)$.
The above integral is therefore a Fourier transform w.r.t. $\xi_3$.
Let $g\in 
C^\infty(\overline{TB})$ satisfy $\p_{\bar z} g=0$ and 
be independent of $y_3$. Then taking the inverse Fourier transform we write
\begin{align*}
0 &= \int_{T_{y_3}} g \p_{\ov{z}} e^{\Psi} dy_1 dy_2 \\
&= \int_{T_{y_3}}  \p_{\ov{z}} (ge^{\Psi})dy_1 dy_2,
\end{align*}
where $T_{y_3 }:=TB\cap \Pi_{y_3}$ and 
$\Pi_{y_3}=\{(y_1,y_2,y_3):(y_1,y_2)\in\R^2\}$. 
Notice that the boundary of $T_{y_3}$ is smooth.
Multiplying the above by $2i$ and
using Stokes' theorem we get that
\begin{align}
0 &= 
2i \int_{T_{y_3 }} \p_{\ov{z}} (ge^{\Psi})dy_1 dy_2\nonumber  \\
&= \int_{T_{y_3 }}  \Curl (ge^{\Psi}, i ge^{\Psi},0) \cdot n dy_1 dy_2\nonumber  \\
&= \int_{\p T_{y_3 }} (ge^{\Psi}, i ge^{\Psi},0) \cdot dl \nonumber \\
&= \int_{\p T_{y_3}} ge^{\Psi}dz, \label{eq_orth_hol_g}
\end{align}
for all holomorphic functions $g\in C^\infty(\overline{T_{y_3}})$.

Next we shall show that \eqref{eq_orth_hol_g} implies that there exists a 
nowhere vanishing holomorphic function $F\in C(\overline{T_{y_3}})$ such that
\begin{equation}
\label{eq_log}
F|_{\p T_{y_3}}=e^{\Psi}|_{\p T_{y_3}}. 
\end{equation}
 
To this end, we define $F$ to be   
\begin{align*}
F(z)=\frac{1}{2\pi i}\int_{\p T_{y_3}} 
\frac{e^{\Psi(\zeta)}}{\zeta-z}d\zeta,  \quad z\in\C\setminus\p T_{y_3}. 
\end{align*}
The function $F$ is holomorphic away from $\p T_{y_3}$.  As $e^{\Psi}$ is Lipschitz, 
we know because of the Plemelj-Sokhotski-Privalov formula (see e.g. \cite{Kress1}), that  
\begin{equation}
\label{eq_PSP_formula}
\lim_{z\to z_0,z\in T_{y_3}} F(z)-\lim_{z\to z_0,z\notin 
T_{y_3}}F(z)=e^{\Psi(z_0)},\quad z_0\in \p 
T_{y_3}.
\end{equation} 
Now the function $\zeta\mapsto (\zeta-z)^{-1}$ is holomorphic on $T_{y_3}$ when 
$z\notin T_{y_3}$. By choosing $g(\zeta)= (\zeta-z)^{-1}$ in \eqref{eq_orth_hol_g},
we get therefore that $F(z)=0$, 
when $z\notin T_{y_3}$.  Hence, the second limit in 
\eqref{eq_PSP_formula} vanishes, and therefore, $F$ is holomorphic function on 
$T_{y_3}$, such that \eqref{eq_log} holds. 

Next we show that $F$ is non-vanishing in $T_{y_3}$.  When doing so, let  
$\p T_{y_3}$ be parametrized by $z=\gamma(t)$, and $N$ be the number of zeros 
of $F$ in $T_{y_3}$. Then by the argument principle, we get
\[
N=\frac{1}{2\pi i}\int_{\gamma}\frac{F'(z)}{F(z)}dz=\frac{1}{2\pi 
i}\int_{F \circ \gamma}\frac{1}{\zeta}d\zeta=\frac{1}{2\pi 
i}\int_{e^{\Psi \circ \gamma}} \frac{1}{\zeta}d\zeta= 0.
\]
To see that the  last integral is zero, notice that this the winding number of the 
path $e^{\Psi \circ \gamma}$. 
And that $e^{\Psi(\gamma(t))}$ is homotopic to 
the constant contour $\{1\}$, with the homotopy given by 
$e^{s\Psi(\gamma(t))}$, $s\in [0,1]$. 

Next, since $F$ is a non-vanishing holomorphic function on $T_{y_3}$ and 
$T_{y_3}$ is simply connected, it admits a holomorphic logarithm. Hence,  
\eqref{eq_log} implies that 
\[
(\log F)|_{\p T_{y_3}}=\Psi|_{\p T_{y_3}}.
\]
Because $\log F=\Psi$ is continuous on $\p T_{y_3}$, we have by the Cauchy theorem, 
\[
\int_{\p T_{y_3}} g \Psi dz=\int_{\p T_{y_3}} g 
\log F dz=0,
\]
where $g\in C^\infty(\overline{T_{y_3}})$ is  an arbitrary function such that 
$\p_{\bar z}g=0$.  
Using Stokes' formula as in \eqref{eq_orth_hol_g} allows us to write this as 
\[
\int_{T_{y_3}} g\p_{\bar z}\Psi dy_1 dy_2=0. 
\]
Taking the Fourier transform with respect to $y_3$, we get
\[
\int_{T(B)} e^{iy\cdot \xi} g\p_{\bar z}\Psi dy=0,
\]
for all $\xi=(0,0,\xi_3)$, $\xi_3\in\R$.   Hence, returning back to the 
$x$ variable, we obtain that 
\[
(\gamma_1 +i\gamma_2)\cdot \int_{B} e^{i x\cdot  \xi}g(x) \nabla \Psi(x) dx=0,
\]
where $g\in C^\infty(\overline{B})$ is such that $(\gamma_1 + 
i\gamma_2)\cdot \nabla g=0$ in $B$. 

Using \eqref{eq_amplitude_sum},  we finally get
\begin{equation}
\label{eq_identity_main_2_new_zero}
(\gamma_1 + i\gamma_2)\cdot \int_{B} (\tilde A_2-\tilde 
A_1)g(x)e^{ix\cdot\xi}dx=0. 
\end{equation}
Setting $g=1$, we obtain \eqref{eq_identity_main_2_new}.  

\end{proof}

By replacing the vector $\gamma_2$ by $-\gamma_2$ in 
\eqref{eq_identity_main_2_new}, we see that 
\begin{equation}
\label{eq_plane_2}
(\gamma_1 -i\gamma_2)\cdot \int_{B} (\tilde A_2 -\tilde 
A_1)e^{ix\cdot\xi}dx=0. 
\end{equation}
Hence,  \eqref{eq_identity_main_2_new} and \eqref{eq_plane_2} imply that
\begin{equation}
\label{eq_plane_3}
\gamma\cdot \int_{B} (\tilde A_2 -\tilde A_1)e^{ix\cdot\xi}dx=0,
\end{equation}
for all $\gamma\in \textrm{span}\{\gamma_1,\gamma_2\}$ and all $\xi\in \R^{3}$ 
such that 
\eqref{eq_gamma-assum} and  \eqref{eq_gamma-rest} hold.

In the proof of the next Proposition  we see that
\eqref{eq_plane_3} is actually a condition for having 
$\Curl (\tilde A_1 - \tilde A_2) = 0$. This is therefore the 
last step in proving that the DN-map determines the curl of the magnetic
potential. 

\begin{prop}
\label{prop_curl_thm_2}
It follows from \eqref{eq_plane_3} that
\begin{equation}
\label{eq_curl_2}
\nabla \times \tilde A_1= \nabla \times \tilde A_2\quad\textrm{in}\quad 
B. 
\end{equation}
\end{prop}

\begin{proof}

Assume that  $\xi\in\R^3$ is not on the line $L:=(0,0,t)$, $t \in \R$. Then the 
vectors $\gamma_1$ and $\gamma_2$ given by
\begin{alignat}{2}
\label{eq_vectors_special}
\tilde \gamma_1 &:= (-\xi_2, \xi_1,0) , \quad
&\gamma_1:=\tilde \gamma_1/|\tilde \gamma_1|, \nonumber \\
\tilde \gamma_2 &:= \xi \times \gamma_1, \quad
&\gamma_2:=\tilde \gamma_2/|\tilde \gamma_2|,
\end{alignat}
where $\xi \times \gamma_1$ stands for the cross product,
satisfy \eqref{eq_gamma-rest} and \eqref{eq_gamma-assum}.
Thus, for any vector $\xi\in\R^3 \setminus L$,
\eqref{eq_plane_3} says that
\begin{equation}
\label{eq_plane_4}
\gamma \cdot  v(\xi)=0, \quad v(\xi):=\widehat{\tilde 
A_2\chi}(\xi)-\widehat{\tilde A_1\chi}(\xi),
\end{equation}
for all $\gamma\in \textrm{span}\{\gamma_1,\gamma_2\}$. Here $\chi$ is the 
characteristic function of the set $B$.  
For any vector $\xi\in \R^3$, we have the following decomposition,
\[
v(\xi)=v_{\xi}(\xi)+v_\perp(\xi),
\] 
where $\Re v_{\xi}(\xi)$, $\Im v_{\xi}(\xi)$ are multiples of $\xi$, and  $\Re 
v_\perp(\xi)$, $\Im v_\perp(\xi)$ are orthogonal to $\xi$.  Now we have $\Re 
v_\perp(\xi), \Im v_\perp(\xi)\in \textrm{span}\{\gamma_1,\gamma_2\}$, and 
therefore, it follows from \eqref{eq_plane_4} 
that $v_\perp(\xi)=0$, 
for all $\xi\in \R^3\setminus L$. 

Hence, 
$v(\xi)=\alpha(\xi)\xi$,
so that that 
\[
\xi\times v(\xi)=0,
\]
for all $\xi\in \R^3\setminus L$, and thus, everywhere, 
by the analyticity of the Fourier transform.  Taking the inverse Fourier 
transform, we obtain \eqref{eq_curl_2}.
\end{proof}

\section{Determining the electric potential}

In order to complete the proof  of Theorem \ref{thm_2_inverse}, we need 
to show that the electric potential is also determined by the
DN-map. Again we assume that 
$A_j,q_j$ and $\Gamma_j$, $j=1,2$ 
are as in Theorem \ref{thm_2_inverse} and that the DN-maps satisfy
\eqref{eq_data_inv}, and hence  \eqref{eq_data_inv_1}.

Since $B$ is simply connected,
it follows from the Helmholtz decomposition of $\tilde A_1-\tilde A_2$  and \eqref{eq_curl_2} that 
there exists $\psi\in C^{1,1}(\ov{B})$ with $\psi=0$ near $\p B$ 
such that 
\[
\tilde A_1=\tilde A_2+\nabla \psi\quad\textrm{in}\quad B. 
\]
We extend $\psi$ to a function of class $C^{1,1}$ on all of $\R^3$ such that 
$\psi=0$ on $\R^3\setminus\ov{B}$. Then 
\[
\tilde A_1=\tilde A_2+\nabla \psi\quad\textrm{in}\quad \R^3. 
\]
Since $\tilde{\Gamma}_j \subset \CHS \setminus \ov{N}$, $j=1,2$, and $\CHS \setminus \ov{N}$
is connected, we can assume that $\psi=0$ on $\HS \setminus N$ and hence,
we have that $\psi=0$ on $\tilde \Gamma_1\cup\tilde \Gamma_2$. 
It follows then
from Lemma \ref{gauge_inv} part (\textit{i}) and \eqref{eq_data_inv_1}
that for all $f$ with 
$\supp(f)\subset\tilde \Gamma_2$,
\[
\Lambda_{A_1,q_1}(f)|_{\tilde \Gamma_1}=
\Lambda_{A_2,q_2}(f)|_{\tilde \Gamma_1}=\Lambda_{A_2+\nabla 
\psi,q_2}(f)|_{\tilde \Gamma_1}=\Lambda_{A_1,q_2}(f)|_{\tilde \Gamma_1}.
\]
Therefore, we can apply Proposition \ref{identity} with $A_1=A_2$ and get
\begin{equation}
\label{eq_recovering_q}
\int_{N}(q_1-q_2)u_1\ov{u_2}dx=0,
\end{equation}
for all $u_1\in W_1(N)$ and $u_2\in W_2^*(N)$. 

Choosing in \eqref{eq_recovering_q} $u_1$ and $u_2$ as the complex geometric 
optics solutions, given by \eqref{eq_u_1-cgo} and \eqref{eq_u_2-cgo}, passing to
$\Bm$, and 
letting $h\to 0$, we have
\begin{equation}
\label{eq_recovering_q_2}
\int_{B}(\tilde q_1-\tilde 
q_2)e^{ix\cdot\xi}e^{\Phi_1^0(x)+\overline{\Phi_2^0 (x)}}dx=0.
\end{equation}
By Remark \ref{rem_g}
$e^{\Phi_1}$ in the definition \eqref{eq_u_1-cgo} of $u_1$ can be replaced  by 
$ge^{\Phi_1}$ if $g\in C^\infty(\overline{B})$ is a solution of 
\[
(\gamma_1+ i \gamma_2)\cdot \nabla g=0\quad  \textrm{in}\quad B.  
\]
Then \eqref{eq_recovering_q_2} can be replaced by 
\[
\int_{B}(\tilde q_1-\tilde 
q_2)g(x)e^{ix\cdot\xi}e^{\Phi_1^0(x)+\overline{\Phi_2^0 (x)}}dx=0.
\]
Now \eqref{eq_amplitude_sum} has the form,
\[
(\gamma_1 + i\gamma_2)\cdot\nabla (\Phi_1^0+\overline{\Phi_2^0})=0\quad 
\textrm{in}\quad B, 
\]
since we have that $\tilde A_1 = \tilde A_2$.
Thus, we can take $g=e^{-(\Phi_1^0+\overline{\Phi_2^0)}}$ and  obtain that 
\begin{equation}
\label{eq_recovering_q_3}
\int_{B}(\tilde q_1-\tilde q_2)e^{ix\cdot\xi}dx=0,
\end{equation}
for all $\xi\in \R^3$ such that there exist $\gamma_1,\gamma_2\in \R^3$, 
satisfying \eqref{eq_gamma-assum} and  \eqref{eq_gamma-rest}.
Since for any $\xi\in\R^3$ not of the form $\xi=(0,0,\xi_3)$, the vectors, 
given by \eqref{eq_vectors_special}, satisfy \eqref{eq_gamma-assum} and  \eqref{eq_gamma-rest},
we conclude 
that \eqref{eq_recovering_q_3} holds for all  $\xi\in\R^3$ 
except those of the form $\xi=(0,0,\xi_3)$, and therefore, by analyticity of the Fourier 
transform, for all $\xi\in \R^3$. Hence,  $q_1=q_2$ in $\Bm$. 
This completes the proof of Theorem \ref{thm_2_inverse}.

\section{Appendix}

\subsection{Magnetic Green's formulas}
\label{sec:EGF}

Let us first recall, following \cite{DosSantos1},  the standard Green formula 
applied to the magnetic Schr\"odinger operator.  

\begin{lem} \label{MagGFI} 
Suppose that $\Omega \subset \R^3$ is open and bounded, with piecewise $C^{\,1}$
boundary. Let $A \in W^{1,\infty}(\Omega,\R^3)$ and  $q \in 
L^{\infty}(\Omega)$. Then we have, 
\begin{align*}
& \quad (\Ld u, v)_{L^2(\Omega)} - (u , L_{A,\ov{q}} v)_{L^2(\Omega)} \\
&= (u , (\partial_n + iA\cdot n)
v)_{L^2(\partial \Omega)}-  ((\partial_n + iA\cdot n) u, v)_{L^2(\partial 
\Omega)},
\end{align*}
for all $u,v \in H^1(\Omega)$, with $\Delta u, \Delta v \in L^2(\Omega)$,
where $n$ is the exterior unit normal to $\p \Omega$.
\end{lem}

We shall also need a version of the above result where $\Omega$ 
is replaced by $\HS$. We shall then need to put some restrictions on $v$ and 
$u$, 
because $\HS$ is unbounded. To this end we assume that $u$ and $v$ are 
solutions to
the Helmholtz equation outside some compact set, 
that obey some form of radiation condition. To be 
precise, let $A \in \Wc^{1,\infty}(\HS,\R^3)$,  $q \in
\Lc^{\infty}(\HS)$, and let 
$u\in H^2_{\textrm{loc}}(\ov{\HS})$ be such that 
\[
(L_{A,q}-k^2) u=0\quad \textrm{in}\quad \HS,
\]
$\supp(u|_{\p\HS})$ is compact, and $u$ is outgoing.  Assume also that $v\in 
H^2_{\textrm{loc}}(\ov{\HS})$ satisfies 
\[
(L_{A,\ov{q}}-k^2)v\in L^2_{\textrm{comp}}(\CHS), 
\]
$\supp(v|_{\p \HS})$ is compact, and $v$ is incoming.

\begin{lem} \label{MagGFII} 
With $u$ and $v$ as above, we have 
\begin{equation}
\label{eq:MagGFIIeq}
\begin{aligned}
& \quad ((\Ld-k^2) u, v)_{L^2(\HS)} - (u , (L_{A,\ov{q}}-k^2) 
v)_{L^2(\HS)} \\
&= (u , (\partial_n + iA\cdot n)
v)_{L^2(\p\HS)}-  ((\partial_n + iA\cdot n) u, v)_{L^2(\p\HS)}.
\end{aligned}
\end{equation}
\end{lem}

\begin{proof}
Let $B_R := B(x_0,R)$ be an open ball in $\R^3$ of radius $R$,  and 
choose $R>0$ large enough so that 
\[
\supp(A),\supp(q)\subset B_R. 
\]
Set $\Omega = \HS \cap B_R$.  By Lemma \ref{MagGFI},  we know that
\begin{align*}
& \quad ((\Ld-k^2) u, v)_{L^2(\Omega)} - (u , (L_{A,\ov{q}}-k^2) 
v)_{L^2(\Omega)} \\
&= (u , (\partial_n + iA\cdot n)
v)_{L^2(\partial \Omega)}-  ((\partial_n + iA\cdot n) u, v)_{L^2(\partial 
\Omega)}. 
\end{align*}
Thus, to obtain \eqref{eq:MagGFIIeq} we need to show that 
\begin{equation}
\label{eq_green_pr1}
\int_{\p B_R\cap \HS} (u  \ov{\partial_n  v}  - \partial_n u\ov{v} ) d S_R 
\to 0, \quad R\to \infty.
\end{equation}
Let us rewrite the left hand side of the above as follows, 
\begin{align*}
 \int_{\p B_R\cap \HS}(\partial_n \ov{v} -ik\ov{v})udS_R - \int_{\p B_R\cap 
\HS}(\partial_n u -iku)\ov{v}dS_R. 
\end{align*}
We show that first term goes to zero as $R \to \infty$. The second term can be
handled in the same way. Applying Cauchy-Schwarz gives
\begin{align*}
\bigg|\int_{\p B_R\cap \HS}(\partial_n \ov{v} -ik\ov{v})u dS_R\bigg|^2
\leq \int_{\p B_R\cap \HS} |\partial_n \ov{v} -ik\ov{v}|^2dS_R \int_{\p B_R\cap 
\HS} | u |^2 dS_R.
\end{align*}
Here the first integral goes to zero, since 
$\ov{\partial_n \ov{v} -ik\ov{v}} = \partial_n v +ikv$ and $|\partial_n v
+ikv|^2$ is $o(1/r^2)$ as $r=|x|\to \infty$, since $v$ is incoming. 
We conclude the proof by showing that the second integral is bounded as $R \to
\infty$.

To this end we let $R_2>R_1$, 
where $R_1$ is such that
\[
    \supp(A),\supp(q)\subset B(x_0,R_1),
\]
and set $B_j := B(x_0,R_j)$, $j=1,2$ and $U:= (B_2 \setminus B_1) \cap \HS$. 
We multiply the Sommerfeld condition \eqref{eq:SRC} by its complex conjugate and integrate over
$\p B_2 \cap \HS$, which gives
\begin{align}
&\int_{\p B_2 \cap  \HS}  \big(k^2|u |^2 + |\partial_n u |^2 + 2k \Im(u \partial_n \ov{u} 
) \big)dS  \nonumber\\
=&\int_{\p B_2 \cap  \HS} |\p_n u-iku|^2dS \nuoli 0,
\label{eq:lim1}
\end{align}
as $R_2 \to \infty$ and where $n$ is the outer unit normal vector to $B_2$.

By Green's formula we have on the other hand that
\begin{align}
\int_{\p U} u \p_n \ov{u} - \ov{u} \p_nu = \int_{U} 
u\Delta \ov{u} - \Delta\ov{u}u  = 0.
\label{eq:gf2}
\end{align}
We may assume that $u|_{\p U \cap \p\HS} = 0$,
since $\supp(u|{ \p \HS})$ is compact. We can thus write \eqref{eq:gf2} as 
\begin{align}
\int_{\p B_2 \cap \HS}  \Im(u \partial_n \ov{u}) 
= \int_{\p B_1 \cap \HS}  \Im(u \partial_n \ov{u}).
\end{align}
But this implies that the $\int |u|^2$ and $\int |\p_nu|^2$ terms stay bounded in
the limit \eqref{eq:lim1}.
\end{proof}

\subsection{The unique continuation principle}

In this work we make heavy use of the so called \textit{unique
continuation principle}. The unique continuation principle can be seen as an 
extension of the  familiar property that an analytic function that 
is zero on some
open set is identically zero.

Let $\Omega\subset\R^n$ be an open connected set, and let
\[
Pu=\sum_{i,j=1}^n a_{ij}(x)\p_i\!\p_j\!u+\sum_{i} b_i(x)\p_i\!u+c(x)u. 
\]
Here $a_{ij}\in C^1(\ov{\Omega})$ are real-valued, $a_{ij}=a_{ji}$, and there 
is $C>0$ so that 
\[
\sum_{i,j=1}^n a_{ij}(x) \xi_i \xi_j \geq C|\xi|^2, \quad x \in 
\ov{\Omega},\quad \xi\in\R^n.
\]
Furthermore, 
$b_i\in L^\infty(\Omega,\C)$ and $c\in L^\infty(\Omega, \C)$. We have the 
following result, see  \cite{Choulli} and \cite{Leis1}.

\begin{thm}\label{UCP} 
Let $u\in H^2_{\emph{loc}}(\Omega)$ be such that $Pu=0$ in $\Omega$. Let 
$\omega\subset\Omega$ be open non-empty. If $u=0$ on $\omega$, then $u$ 
vanishes identically in $\Omega$. 

\end{thm}

\begin{cor} \label{UCP_boundary}  Assume that $\p \Omega$ is of class $C^2$. 
Let $\Gamma\subset \p \Omega$ be open non-empty. Let $u\in H^2(\Omega)$ be such 
that $Pu=0$ in $\Omega$. Assume that 
\[
u=\mathcal{B}_{\nu}u=0\quad \textrm{on}\quad \Gamma.
\]
Here $\mathcal{B}_{\nu}u$ is  the conormal derivative of $u$, given by
\[
\mathcal{B}_{\nu}u=\sum_{i, j=1}^n \nu_i (a_{ij}\p_{j} u)|_{\p \Omega}\in 
H^{1/2}(\p \Omega).
\]
Then $u$ vanishes identically in $\Omega$. 

\end{cor}

\section*{Acknowledgements}  

I like to thank my thesis advisors Katya Krupchyk and Lassi P\"aiv\"arinta for all the advice 
I have recived during the preparation of this paper. 
This research has been supported by the Academy of Finland (project 255580).

\bibliographystyle{plain}
\bibliography{ref}

\end{document}